\documentclass[a4paper,11pt]{amsart}
\usepackage{amsmath,amssymb,colordvi}
\usepackage{verbatim}
\newtheorem{theorem}{Theorem}
\newtheorem{corollary}[theorem]{Corollary}
\newtheorem{lemma}[theorem]{Lemma}

\newtheorem{proposition}[theorem]{Proposition}
\newtheorem{definition}[theorem]{Definition}

\newtheorem{remark}[theorem]{\it Remark}

\usepackage{amsthm}

\usepackage{mathrsfs}

\usepackage{amsmath}

\usepackage{color}

\usepackage{amssymb}

\def\NN{\mathbb N}
\def\N{\mathbb N}

\def\R{\mathbb R}
\def\C{\mathbb C}
\def\RR{\mathbb R}

\begin{document}

\title[Laguerre expansions for $C_0$-semigroups and resolvent operators]
{$C_0$-semigroups  and resolvent operators approximated  by Laguerre expansions}

\author{Luciano Abadias}
\address{Departamento de Matem\'aticas, Instituto Universitario de Matem\'aticas y Aplicaciones, Universidad de Zaragoza, 50009 Zaragoza, Spain.}

\email{labadias@unizar.es}

\author{Pedro J. Miana}
\address{Departamento de Matem\'aticas, Instituto Universitario de Matem\'aticas y Aplicaciones, Universidad de Zaragoza, 50009 Zaragoza, Spain.}
\email{pjmiana@unizar.es}

\thanks{Authors have been partially supported by  Project MTM2013-42105-P, DGI-FEDER, of the MCYTS; Project E-64, D.G. Arag\'on; and  Project UZCUD2014-CIE-09, Universidad de Zaragoza.}


\dedicatory{To Laura and Pablo}


\begin{abstract} In this paper  we introduce  Laguerre expansions to approximate  vector-valued functions expanding on the well-known scalar theorem.
We apply this result to approximate
$C_0$-semi\-groups and resolvent operators in abstract Banach spaces.  We  study certain Laguerre functions, its Laplace transforms and the convergence of Laguerre series in Lebesgue spaces. The concluding section of this paper is devote to consider some examples of $C_0$-semigroups: shift, convolution and  holomorphic semigroups where some of these results are improved.

\end{abstract}

\date{}

\maketitle

\section{Introduction}

Representations of functions through series of orthogonal polynomials such as Legendre, Hermite or Laguerre are well known in mathematical analysis and applied mathematics. They allow us to approximate functions by series of orthogonal polynomials on different types of convergence: a pointwise way, uniformly, or in Lebesgue norm. Two classical monographs where we can find this kind of results are \cite[Chapter 4]{Lebedev} and \cite[Chapter IX]{Szego}. In this paper we are concentrated on Laguerre expansions.

 Rodrigues' formula gives a representation of generalized Laguerre polynomials, $$L_n^{(\alpha)}(t)=e^{t}\frac{t^{-\alpha}}{n!}\frac{d^n}{dt^n}(e^{-t}t^{n+\alpha}), \qquad t\in \RR,$$ for $\alpha\in \R$ and $n\in\N\cup\{ 0 \}.$  Here we present an interesting theorem which appears in \mbox{\cite[ Sec. 4.23]{Lebedev},} and whose statement was originally proved by J.V. Uspensky in \cite{Uspensky}. As we prove in Theorem \ref{upesvect}, this result also holds for vector-valued functions in abstract Banach spaces.

\begin{theorem} \label{upes} Let $\alpha>-1$ and $f:(0,\infty)\to\C$ be a differentiable function such that the integral $\int_{0}^{\infty}e^{-t}t^{\alpha}|f(t)|^2\,dt$ is finite, then the series $\displaystyle\sum_{n=0}^{\infty}c_n(f)\,L^{(\alpha)}_n(t),$ with $$c_n(f):=\frac{n!}{\Gamma(n+\alpha+1)}\int_{0}^{\infty}e^{-t}t^{\alpha}f(t)L^{(\alpha)}_n(t)\,dt,$$ converges pointwise to $f(t)$ for $t>0$.
\end{theorem}

There exists a large amount of  results about Laguerre expansions: for example, Laguerre expansions of analytic functions are considered in \cite{Rusev};   the decay of coefficients is also studied to show properties of the function defined by the Laguerre expansions in \cite{Weniger}; and the algebraic structure related to the Laguerre expansions is treated in detail in \cite{Kanjin}.

In particular,  Theorem \ref{upes} is applied to the function $e_a$ (where $e_a(t):=e^{-at})$  to obtain that \begin{equation}\label{expo2} e^{-a t}=\displaystyle\sum_{n=0}^{\infty}\frac{a^n}{(a +1)^{n+\alpha+1}}L_n^{(\alpha)}(t), \qquad a>0\end{equation} and the Laguerre expansion converges pointwise for $t>0$.

Through Laguerre polynomials, Laguerre functions are defined by $$\mathscr{L}_n^{(\alpha)}(t):=\sqrt{\frac{n!}{\Gamma(n+\alpha+1)}}t^{\frac{\alpha}{2}}e^{-\frac{t}{2}}L_n^{(\alpha)}(t),\quad t>0,$$ for $\alpha>-1$.  These functions form an orthonormal basis on the Hilbert space $L^2(\R_+).$ Furthermore, let $f$ be in $L^p(\R_+),$ $\frac{4}{3}<p<4,$ and $\displaystyle{a_k(f):=\int_0^{\infty}\mathscr{L}_k^{(\alpha)}(t)f(t)\,dt}$ for \mbox{$k\in \NN\cup\{0\}.$} Then $\lVert S_n(f)-f \rVert_p\to 0$ as $n\to\infty,$ with $$S_n(f)(t):=\displaystyle\sum_{k=0}^{n}a_k(f)\mathscr{L}_k^{(\alpha)}(t),\qquad t>0,$$ see \cite[Theorem 1]{Askey}.

On the other hand, a $C_0$-semigroup $(T(t))_{t\geq 0}$ is a one parameter family of linear and bounded operators on a Banach space $X$ which may be interpreted, roughly speaking, as $(e^{-tA})_{t\geq 0}$. The (densely defined) operator $-A,$ defined by $$-Ax:=\displaystyle\lim_{t\to 0^+}\frac{T(t)x-x}{t},\quad \hbox{when the limit exists} ,\, x\in D(A),$$ is called the infinitesimal generator of $C_0$-semigroup, see more details in the fourth section and in monographs \cite{ABHN, Nagel}. It seems natural to consider the formula (\ref{expo2}) in the context of $C_0$-semigroups; we prove that
 $$
 T(t)x=\displaystyle\sum_{n=0}^{\infty}A^n(A+1)^{-n-\alpha-1}x\,L^{(\alpha)}_n(t), \qquad x\in D(A),
$$
 in Theorem \ref{loc} (iii). The rate of convergence of $ \displaystyle\sum_{n=0}^{m}A^n(A+1)^{-n-\alpha-1}x\,L^{(\alpha)}_n(t)$ to $T(t)x$ when $m\to\infty$ is also estimated in  Theorem \ref{rate}.

In semigroup theory, there are different types approximations of $C_0$-semigroups as Trotter-Kato, Yosida, or Euler approximations. In a recent paper (\cite{Gomilko}), a new functional calculus is introduced  to improve some rate of these approximations. Also Pad\'{e} approximations and the rate of convergence to the semigroup have been treated in \cite{BT, Egert, Neubrander}.  Approximation of $C_0$-semigroups by resolvent series  has been considered in \cite{Grimm} using some analytic functional calculus. In this work, authors also work on the Laguerre polynomials $(L^{(-1)}_n)_{n\ge 0}$, however their approach is completely different to ours because they consider other different approximations,  compare \cite[Theorem 5.1]{Grimm} and Theorem \ref{loc} (iii).

The paper is organized as follows. In the second section, we consider the functions $t\mapsto\frac{n!}{\Gamma(n+\alpha+1)}e^{-t}t^{\alpha}L^{(\alpha)}_n(t)$ (for $\alpha>-1$ and $n\in \NN\cup\{0\}$) which have a key role in Theorem \ref{upes}. They satisfy interesting properties (similar to Laguerre polynomials, see Proposition \ref{alge} and \ref{propLag}). We estimate its $p$-norm in Theorem \ref{main2} and also  the $p$-norm of its Laplace transform in Theorem \ref{phi}.

In the third section, we consider certain Laguerre expansions in Lebesgue space $L^p(\R_+)$ and abstract Banach spaces. In particular we give the Laguerre expansion for the fractional semigroup and for the exponential function $e_a$  in $L^1(\R_+)$ (Theorem \ref{fract}).  To finish this section, we show a vector-valued version of Theorem \ref{upes} on an abstract Banach space $X$.

Main applications of Theorem \ref{upesvect} appear in the fourth and fifth section. In Theorem \ref{loc}, we express $C_0$-semigroup  and resolvent operators though Laguerre expansions. In Theorem \ref{appr}, we apply this Laguerre expansion  to express the resolvent semigroup subordinated  to the original $C_0$-semigroup by a series representation. In Theorem \ref{rate}, we give the rate of the Laguerre expansion to the $C_0$-semigroup, which depends on the regularity of the initial data.

In the last section, we present some examples of $C_0$-semigroups and its Laguerre expansion: shift semigroup, convolution semigroups and in particular, Gauss and Poisson semigroups in Lebesgue spaces. For some differentiable and analytic semigroups, some previous results  are improved.


Throughout the paper, we work on functions and operators defined in $\R_+,$ but there are also other results on expansions for functions defined in $\R$ using Hermite polynomials. Hermite polynomials are defined by Rodrigues' formula $$H_n(t)=(-1)^n e^{t^2}\frac{d^n}{dt^n}(e^{-t^2})\ \text{for }t\in\R\ \text{and}\ n\geq 0.$$ A similar result to Theorem \ref{upes} holds for functions defined on $\R$ and involves Hermite polynomials,  see \cite[Sec. 4.15]{Lebedev}, in particular for $\lambda\in\C,$ $$e^{\lambda t}=e^{\frac{\lambda^2}{4}}\displaystyle\sum_{n=0}^{\infty}\frac{\lambda^n}{2^n n!}H_n(t), \qquad t\in \RR.$$ In the preprint \cite{AM}, we  work on vector-valued approximations defined by Hermite polynomials, and its applications to    $C_0$-groups and cosine operator families.


$\ $

\noindent {\bf Notation}. We write $\R_+:=[0,\infty).$ Given $1\leq p <\infty,$ let $L^p(\R_+)$ be the set of Lebesgue $p$-integrable functions, that is, $f$ is a measurable function and $$\Vert f \Vert_p:=\left(\int_0^{\infty}|f(t)|^p\,dt\right)^{\frac{1}{p}}<\infty;$$ for $p=2$, remind that $L^2(\R_+)$ is a Hilbert space with $\langle\,\,,\,\, \rangle$ the usual inner product, and $L^{\infty}(\R_+)$ the set of essential bounded Lebesgue functions with the norm $\lVert f\rVert_{\infty}:=\displaystyle\hbox{esssup}_{t\in\R_+}|f(t)|.$
We call  $C_0(\R_+)$ the set of continuous functions defined in $[0,\infty)$ such that $\displaystyle\lim_{t\to \infty}f(t)=0,$ with the norm $\lVert \ \rVert_{\infty};$ and $\mathcal{H}_0(\C_+)$ the Banach space given by the set of bounded holomorphic functions defined in $\C_+=\{z\in\C\,:\,{\Re}z>0\}$ such that $\displaystyle\lim_{z\to \infty}f(z)=0,$ with the norm $||| f|||_{\infty}:=\displaystyle\sup_{z\in\C_+}|f(z)|.$ Furthermore, it is important to consider the following equivalence, which appears several times throughout the paper: for $\alpha>-1$ there exist $c_1 <c_2$ such that
\begin{equation} \label{main}
{c_1\over n^\alpha}\le {n!\over \Gamma(n+\alpha +1)}\le {c_2\over n^\alpha}, \qquad n\in \N.
\end{equation}

\section{ Laguerre  functions and its Laplace transform}

\setcounter{theorem}{0}
\setcounter{equation}{0}
Generalized Laguerre polynomials $\{ L_n^{(\alpha)} \}_{n\ge 0}$ ($\alpha >-1$) are given  by
$$
L_n^{(\alpha)}(t)=\sum_{k=0}^n(-1)^{k}{n+\alpha\choose n-k}{t^k\over k!}, \qquad t\ge 0;
$$
in particular $L_0^{(\alpha)}(t)=1$, $L_1^{(\alpha)}(t)=-t+\alpha+1$  and
$\displaystyle{L_2^{(\alpha)}(t)={t^2\over 2}-(\alpha+2)t+{(\alpha+2)(\alpha+1)\over 2}
}$. They satisfy the following condition of orthogonality: $${n!\over \Gamma(n+\alpha +1)}\int_0^{\infty}e^{-t}t^{\alpha}L_n^{(\alpha)}(t)L_m^{(\alpha)}(t)dt=\delta_{n,m},$$ for $n,m=0,1,2,\ldots,$ with $\delta_{n,m}$ is the Kronecker delta. The generalized Laguerre polynomials are solutions of second order differential equation $$\label{Laguerresoleq}ty''+(\alpha+1-t)y'+ny=0.$$

Now, we provide a list of  well-known identities which  these polynomials satisfy,
\begin{eqnarray*}
nL_n^{(\alpha)}(t)&=&(n+\alpha)L_{n-1}^{(\alpha)}(t)-tL_{n-1}^{(\alpha+1)}(t),\cr
tL_{n}^{(\alpha+1)}(t)&=&(n+\alpha)L_{n-1}^{(\alpha)}(t)-(n-t)L_n^{(\alpha)}(t),\cr
nL_n^{(\alpha)}(t)&=&(\alpha+1-t)L_{n-1}^{(\alpha+1)}(t)-tL_{n-2}^{(\alpha+2)}(t),\cr
0&=&(n+1)L_{n+1}^{(\alpha)}(t)+(t-\alpha-2n-2)L_{n}^{(\alpha)}(t)+(n+\alpha)L_{n-1}^{(\alpha)}(t),
\end{eqnarray*}
 see for example \cite{Lebedev} and \cite{Szego}.

Muckenhoupt estimates for Laguerre functions $\mathscr{L}_n^{(\alpha)}$  are well known in the classical theory of orthogonal expansions, see \cite{Muckenhoupt} and \cite[Lemma 1.5.3]{Thangavelu}. As a consequence of these estimations,
\begin{eqnarray}\label{muck}
\Vert \mathscr{L}_n^{(0)} \Vert_1&\sim& \sqrt{n}, \qquad n\ge 1,\\
|L_n^{(\alpha)}(\lambda)|&\leq& C_{\lambda}n^{\frac{\alpha}{2}},\quad n\ge n_0,\label{mucklag}
\end{eqnarray}
hold for $\lambda>0,$ see \cite[Lemma 1]{Markett} and \cite[Lemma 1.5.4. (i)]{Thangavelu}. Similar results can be found in  \cite[Lemma 5.1]{Eisner} and \cite[Theorem 8.1]{Gomilko}.

In this paper  we are mainly interested in the following Laguerre functions.
\begin{definition}
For $\alpha\neq -1,-2,-3,\ldots,$  and $n\in \NN\cup\{0\}$, we denote by $\ell_n^{(\alpha)}$ the function defined in $(0,\infty)$ by
$$
\ell_n^{(\alpha)}(t):=\frac{n!}{\Gamma(n+\alpha+1)}t^{\alpha}e^{-t}L_n^{(\alpha)}(t).
$$
\end{definition}
Note that $\displaystyle\ell_n^{(\alpha)}(t)=\frac{1}{\Gamma(n+\alpha+1)}\frac{d^n}{dt^n}(e^{-t}t^{n+\alpha}),$ for $ t> 0,$ with $n=0,1,2,\ldots$ and  $\ell_{0}^{(\alpha)}=I_{\alpha+1} $, where $(I_s)_{s\in\C_+}$ is the fractional semigroup defined by $$I_s(t):=\frac{1}{\Gamma(s)}e^{-t}t^{s-1}, \qquad t>0,\, s\in\C_+.$$ In  \cite[Theorem 2.6]{Sinclair} the fractional semigroup $(I_s)_{z\in\C_+}$ is studied in detail; in particular $I_s\ast I_t=I_{t+s}$ for $s,\, t\in\C_+$ (i.e. algebraic semigroup)  where
 the convolution product of $f\ast g$ is defined by
$$
(f\ast g)(t):=\int_0^tf(t-s)g(s)ds, \qquad f,g\in L^1(\R_+),\quad t\ge 0.
$$
In the case of $\mu\in M(\R_+)$ is a non negative  Borel measures on $\R_+$ of total variation,  the convolution product $f\ast \mu$ is given by
$$
(f\ast \mu)(t):=\int_0^tf(t-s)d\mu(t), \qquad f \in L^1(\R_+),\quad t\ge 0.
$$
The Dirac delta $\delta_0$ verifies that $f\ast \delta_0 =f$ for any $f\in L^1(\R_+)$. We write by $f^{\ast n}=f\ast f^{\ast(n-1)}$ for $n\ge 2$ and $f^{\ast 1}=f$. Laguerre functions $\{ \ell_n^{(\alpha)}\}_{n\ge 0}$ satisfy the following algebraic property (for convolution product $\ast$).

\begin{proposition}\label{alge} For $\alpha>-1$, $n \in \NN$, and $e_1(t):=e^{-t}$ for $t>0$, we have that
\begin{equation}
\label{deco}
 \ell_n^{(\alpha)}= (\delta_0-e_1)^{\ast n}\ast  \ell_0^{(\alpha)}.
\end{equation}
Then the equality $\ell_{n}^{(\alpha)}*\ell_{m}^{(\beta)}=\ell_{n+m}^{(\alpha+\beta+1)}\ $ holds for all $\ \alpha ,\beta>-1 $  and $n,m\ge 0$.

\end{proposition}
\begin{proof} Note that $e_1^{\ast n}=I_n$, we write $I_0:=\delta_0$  and then
\begin{eqnarray*}
(\delta_0-e_1)^{\ast n}\ast\ell_n^{(\alpha)}&=&\sum_{k=0}^n(-1)^k{ n!\over k!(n-k)!}I_k\ast I_{\alpha+1}=\sum_{k=0}^n(-1)^k{ n!\over k!(n-k)!}I_{k+\alpha+1}\\
&=&{n!\over \Gamma(n+\alpha+1)}t^\alpha e^{-t} \sum_{k=0}^n(-1)^k{\Gamma(n+\alpha+1)\over (n-k)!\Gamma(k+\alpha +1)}{t^{k}\over k!}= \ell_n^{(\alpha)},
\end{eqnarray*}
for $\alpha>-1$, $n \in \NN$. The algebraic equality $\ell_{n}^{(\alpha)}*\ell_{m}^{(\beta)}=\ell_{n+m}^{(\alpha+\beta+1)}\ $ for  $\ \alpha ,\beta>-1 $  and $n,m\ge 0$
 is a straightforward consequence of the equality (\ref{deco}) and $\ell_{0}^{(\alpha)}=I_{\alpha+1} $.
\end{proof}

Functions $\{\ell_n^{(\alpha)}\}_{n\ge 0}$ satisfy  recurrence relations, differential equations and some additional identities as the next proposition shows. The proof is left to the reader.

\begin{proposition}\label{propLag} For $\alpha\neq -1,-2,-3,\ldots,$ the family of functions $\{ \ell_n^{(\alpha)} \}_{n\ge 0}$  satisfies:
\begin{itemize}
\item[(i)] $$\ell_n^{(\alpha)}(t)=\ell_{n-1}^{(\alpha)}(t)-\ell_{n-1}^{(\alpha+1)}(t).$$
\item[(ii)] $$(n+\alpha+1)\ell_{n}^{(\alpha+1)}(t)=n \ell_{n-1}^{(\alpha)}(t)-(n-t )\ell_{n}^{(\alpha)}(t).$$
\item[(iii)] $$t \ell_{n}^{(\alpha)}(t)=(\alpha+1-t)\ell_{n-1}^{(\alpha+1)}(t)-(n-1)\ell_{n-2}^{(\alpha+2)}(t).$$
\item[(iv)] $$(n+\alpha+1)\ell_{n+1}^{(\alpha)}(t)+(t-\alpha-2n-1)\ell_{n}^{(\alpha)}(t)+n\ell_{n-1}^{(\alpha)}(t)=0.$$
\item[(v)] $$t\frac{d^2}{dt^2} \ell_{n}^{(\alpha)}(t)+(1-\alpha+t)\frac{d}{dt}\ell_{n}^{(\alpha)}(t)+(n+1)\ell_{n}^{(\alpha)}(t)=0.$$
\item[(vi)] For $k\geq 1$, $\displaystyle{\frac{d^k}{dt^k}\ell_{n}^{(\alpha)}(t)=\ell_{n+k}^{(\alpha-k)}(t).}$

\end{itemize}
\end{proposition}

In the next theorem, we present some estimates of $p$-Lebesgue norm of Laguerre functions $\{\ell^{(\alpha)}_n\}_{n\ge 0}$.

\begin{theorem} \label{main2} Take $p\ge 1$, $\alpha >-1 $ and the set of functions $\{\ell^{(\alpha)}_n\}_{n\ge 0}$. We denote by $\mathcal{Z}(L_n^{(\alpha)})$ the set of zeros of generalized Laguerre polynomial of degree $n$, $L_n^{(\alpha)}.$
\begin{enumerate}

\item[(i)] For $\alpha>0$, and $n\ge 0$, the inequality $\lVert \ell_{n}^{(\alpha)} \rVert_{\infty}\leq\lVert \ell_{n+1}^{(\alpha-1)} \rVert_1$ holds.

\item[(ii)] The set of functions $ \{\ell^{(\alpha)}_n\}_{n\ge 0} \subset L^1(\R_+)$ for $\alpha >{-1}$, and
$$
 \displaystyle\max_{t\in\mathcal{Z}(L_n^{(\alpha)})}|\ell_{n-1}^{(\alpha)}(t)|\le\Vert \ell^{(\alpha)}_n\Vert_{1}\le {C_{\alpha}\over n^{\alpha \over 2}}, \qquad n\ge 1.
$$

 \item[(iii)] The set of function $ \{\ell^{(\alpha)}_n\}_{n\ge 0} \subset L^p(\R_+)$ for $\alpha >-{1\over p}$ and
$
\displaystyle\lVert \ell^{(\alpha)}_n\rVert_{p}\leq  C_{p,\alpha}\ n^{\frac{1}{2}} ,$ for  $ n\ge 1.
$

\end{enumerate}
\end{theorem}
\begin{proof} To show (i), it is sufficient to use that $$|\ell_{n}^{(\alpha)}(t)|\leq \int_t^{\infty}|\frac{d}{ds}\ell_{n}^{(\alpha)}(s)|\,ds\leq\lVert \ell_{n+1}^{(\alpha-1)}\rVert_1,\qquad t>0,$$
where we have applied  Proposition \ref{propLag} (vi). The second inequality in part (ii) is  in \cite[(5.7.16)]{Szego}.  By the part (i), and due to $\displaystyle\lim_{t\to 0+}\ell_{n-1}^{(\alpha+1)}(t)=0=\displaystyle\lim_{t\to \infty}\ell_{n-1}^{(\alpha+1)}(t)$, we obtain that
 \begin{eqnarray*}\lVert \ell_n^{(\alpha)}\rVert_1&\geq& \lVert \ell_{n-1}^{(\alpha+1)} \rVert_{\infty} =\displaystyle\max_{\{ t\in\R_+\ |\ (\ell_{n-1}^{(\alpha+1)})'(t)=0 \}}\vert \ell_{n-1}^{(\alpha+1)}(t)\vert \\
&=&\displaystyle\max_{\{ t\in\R_+\ |\ \ell_{n}^{(\alpha)}(t)=0 \}}\vert\ell_{n-1}^{(\alpha+1)}(t)\vert =\displaystyle\max_{\{ t\in\R_+\ |\ L_{n}^{(\alpha)}(t)=0 \}}\vert \ell_{n-1}^{(\alpha)}(t)\vert,
\end{eqnarray*}
where we have used Proposition \ref{propLag} (vi) and (i) to get the result. The part (iii) is a direct consequence of   Muckenhoupt estimates (\cite{Muckenhoupt} and \cite[Lemma 1.5.3]{Thangavelu}), see also \cite[Lemma 1]{Markett}.
\end{proof}

\begin{remark} {\rm To obtain other estimates for $\lVert \ell_n^{(\alpha)}\rVert_p,$ there exists the possibility of using different bounds of Laguerre polynomials. By \cite[Corollary 2.2]{Du}, we conclude that  there is a  constant $C_{\alpha}>0$ such that $$\lVert \ell_n^{(\alpha)}\rVert_p\leq C_{\alpha}\, n^{2+[\alpha]-\alpha},\qquad \alpha>-1,\quad p\ge 1,$$
 where $[\alpha]$ is the integer part of the real number $\alpha.$ For $\alpha=0$, the bound $ \vert L_n^{(0)}(t)\vert \le e^{t\over 2}$ for $t\ge 0$ and $n\in \NN\cup\{0\}$ appears in \mbox{\cite[(7.21.3)]{Szego},} and then
  $$
  \Vert \ell^{(0)}_n\Vert_p\le\left({p\over 2}\right)^{1\over p}, \qquad p>1.
$$ In \cite[Theorems 7.6.2 and 7.6.4]{Szego}, \cite[Theorem 3]{Duran2} and \cite[Theorem 2]{Love}, other pointwise bounds for Laguerre polynomials are obtained. From those, bounds of $\Vert \ell^{(\alpha)}_n\Vert_p$ might be deduced.
}

\end{remark}

We denote with $\mathcal{L}$ the usual Laplace transform, $\mathcal{L}: L^1(\R_+)\to \mathcal{H}_0(\C_+)$ defined by
 $$
\mathcal{L}(f)(z):=\int_0^\infty e^{-zt}f(t)dt, \qquad f\in L^1(\R_+),\quad z\in\C_+.
$$
Remind that the Laplace transform is a bounded linear operator, $||| \mathcal{L}(f)|||_\infty \le \Vert f\Vert_1$ such that $\mathcal{L}(f\ast g)=\mathcal{L}(f)\mathcal{L}(g)$ for  $f,g\in L^1(\R_+)$. Observe that \begin{equation}\label{LaplaceLagFunc}
\mathcal{L}(\ell_{n}^{(\alpha)})(z)=\frac{z^n}{(z+1)^{n+\alpha+1}}, \qquad z\in\C_{+},
\end{equation}
 see \cite[p.110]{Badii}. For convenience, we write by \begin{equation}\label{definit}\varphi_{n,\alpha}(z):=\frac{z^n}{(z+1)^{n+\alpha+1}}, \qquad z\in \C_+.\end{equation}


\begin{lemma} \label{222} Let $\varphi_{n,\alpha}$ defined by (\ref{definit}) for $n\in \N\cup\{0\}$ and $\alpha>-1$.
\begin{itemize}
\item[(i)] The equality $\varphi_{n,\alpha}'=n\varphi_{n-1,\alpha+2}-(\alpha+1)\varphi_{n,\alpha+1},$ holds $\text{ for }n\geq 1.$

\item[(ii)] For $j\ge 1$, we have that
$$\varphi_{n,\alpha}^{(j)}=\sum_{l=0}^{\min(j,n)}\frac{n!}{(n-l)!}b_{l,\alpha}\,\varphi_{n-l,\alpha+j+l}, $$
where $b_{l, \alpha}$ is a real number dependent on $l$ and $\alpha$.
\end{itemize}
\end{lemma}
\begin{proof}

(i) Note that
$$\varphi_{n,\alpha}'(z)=\frac{nz^{n-1}}{(z+1)^{n+\alpha+1}}-\frac{(n+\alpha+1)z^n}{(z+1)^{n+\alpha+2}}=n\varphi_{n-1,\alpha+2}(z)-(\alpha+1)\varphi_{n,\alpha+1}(z),$$ for $z\in \C_+$ and $n\geq1$. To show the part (ii), first we consider $1\le j< n$. We prove the equality by the inductive method.  For $j=1$, we just prove the equality in   part (i). Now take $j+1\le n$ and again by part (i), we get the following:
\begin{eqnarray*}
&\quad&\varphi_{n,\alpha}^{(j+1)}=\left(\displaystyle\sum_{l=0}^{j}\frac{n!}{(n-l)!}b_{l,\alpha}\,\varphi_{n-l,\alpha+j+l}\right)'=\displaystyle\sum_{l=0}^{j}\frac{n!}{(n-l)!}b_{l,\alpha}\,\varphi_{n-l,\alpha+j+l}'  \cr
&\quad&\quad=\displaystyle\sum_{l=1}^{j+1}\frac{n!}{(n-l)!}b_{l-1,\alpha}\varphi_{n-l,\alpha+j+l+1}-\displaystyle\sum_{l=0}^{j}\frac{n!}{(n-l)!}b_{l,\alpha}(\alpha+j+l+1)\varphi_{n-l,\alpha+j+l+1} \cr
&\quad&\quad=\displaystyle\sum_{l=0}^{j+1}\frac{n!}{(n-l)!}\widetilde{b_{l,\alpha}}\,\varphi_{n-l,\alpha+j+l+1},
\end{eqnarray*}
and the identity is obtained. In the case that $j\geq n$, note that $\varphi'_{0,\alpha}=-(\alpha+1)\varphi_{0,\alpha+1}$,
 \begin{eqnarray*}
&\quad&\varphi_{n,\alpha}^{(j)}=\sum_{l=0}^{n}\frac{n!}{(n-l)!}b_{l,\alpha}\,\varphi_{n-l,\alpha+j+l},
\end{eqnarray*}
and then the equality is obtained.
\end{proof}

For $j\in \N,$ denote by $AC^{(j)}$ the Sobolev Banach space obtained as the completion of
the Schwartz class $\mathcal{S}(\R_+)$ (the set  of restrictions to $[0, \infty)$ of
the Schwartz class $\mathcal{S}(\R)$, see more details in \cite[Definition 1.1]{Du} and \cite{Gale}) in the norm $$\lVert f \rVert_{(j)}:=\frac{1}{(j-1)!}\int_0^{\infty}|f^{(j)}(x)|x^{j-1}\,dx,\qquad f\in \mathcal{S}(\R_+).$$
Note that the following continuous embeddings hold: $(AC^{(j+1)},\lVert\ \rVert_{(j+1)})\hookrightarrow (AC^{(j)},\lVert\ \rVert_{(j)})$ $\hookrightarrow (C_0(\R_+),\lVert\ \rVert_{\infty})$ (see \cite[Proposition 3.1.(i)]{Gale}). These function spaces have been considered by several authors and extended considering  Weyl fractional derivation, instead of the usual derivation in \cite{Gale}, see references and more details therein.

\begin{theorem}\label{phi} Let $\alpha>-1.$ \begin{itemize}

\item[(i)] For $1\leq p<\infty,$ $\varphi_{n,\alpha}\in L^p(\R_+),$ for each $n\geq 0$ whenever $\alpha>\frac{1}{p}-1, $ and  $$\lVert \varphi_{n,\alpha} \rVert_p^p=\frac{\Gamma(np+1)\Gamma(p\alpha+p-1)}{\Gamma(p(n+\alpha+1))}. $$

\item[(ii)] For each $n\geq 1,$ there exists $N_\alpha> M_\alpha>0$ such that ${\frac{M_{\alpha}}{n^{\frac{\alpha+1}{2}}}\le |||\varphi_{n,\alpha}|||_{\infty}\le \frac{N_{\alpha}}{n^{\frac{\alpha+1}{2}}}.}$

\item[(iii)] For $j\in \N,$ $\varphi_{n,\alpha}\in AC^{(j)},$ for each $n\ge 0,$ and ${\lVert \varphi_{n,\alpha} \rVert_{(j)}\le\frac{C_{j,\alpha}}{n^{\alpha+1}}}$  for $ n\ge 1.$

\end{itemize}
\end{theorem}
\begin{proof}(i)  Note that $\displaystyle{\int_0^{\infty}|\varphi_{n,\alpha}(t)|^p\,dt=\int_0^{\infty}\frac{t^{np}}{(t+1)^{p(n+\alpha+1)}}\,dt=\beta(np+1,p(\alpha+1)-1)},$ where $\beta$ is the Euler Beta function. (ii) By the formula \eqref{LaplaceLagFunc}, $\varphi_{n,\alpha}\in\mathcal{H}_0(\C_+), $ and  $\displaystyle{|||\varphi_{n,\alpha}|||_{\infty}=\displaystyle\sup_{z\in \C_+}\frac{|z|^n}{|z+1|^{n+\alpha+1}}.}$ Furthermore, $|\varphi_{n,\alpha}(z)|\leq 1$ for all $z\in \C_+.$ We apply the Maximum Modulus Theorem to get that $$|||\varphi_{n,\alpha}|||_{\infty}=\displaystyle\max_{x\in\R}\frac{|ix|^n}{|ix+1|^{n+\alpha+1}}=\displaystyle\max_{x\in\R}\frac{|x|^n}{(x^2+1)^{\frac{n+\alpha+1}{2}}}.$$ We define $g(t):=\frac{t^n}{(t^2+1)^{\frac{n+\alpha+1}{2}}}$ for $t>0.$ We have that $$\displaystyle\max_{t\geq 0}g(t)=(\alpha+1)^{\frac{\alpha+1}{2}}\frac{n^\frac{n}{2}}{(n+\alpha+1)^{\frac{n+\alpha+1}{2}}}.$$ Then  we conclude  the result. (iii) It is enough to show for $n\ge j$. We apply the Lemma \ref{222} (ii) to get that
\begin{eqnarray*}
\lVert \varphi_{n,\alpha}\rVert_{(j)}&\leq &\frac{1}{(j-1)!}\displaystyle\sum_{l=0}^{j}\frac{n!}{(n-l)!}|b_{l,\alpha}|\int_0^{\infty}\frac{t^{n+j-l-1}}{(1+t)^{n+\alpha+j+1}}\,dt \\
&\leq & C_{j,\alpha}\frac{\Gamma(\alpha+j+1)}{(j-1)!\Gamma(n+\alpha+j+1)}\displaystyle\sum_{l=0}^{j}\frac{n!(n+j-l-1)!}{(n-l)!} \\
&\leq & (j+1) \tilde  C_{j,\alpha}\frac{(n+j-1)!}{\Gamma(n+\alpha+j+1)}\le \frac{C_{j,\alpha}}{n^{\alpha+1}},
\end{eqnarray*}
where we use the inequality (\ref{main}).
\end{proof}

\begin{remark}\label{whittaker} {\rm Observe that if $\alpha>0,$ $\varphi_{n,\alpha}\in L^1(\R_+)$ for each $n\geq 0,$ and $$\mathcal{L}(\varphi_{n,\alpha})(\lambda)=n!\lambda^{\frac{\alpha-1}{2}}e^{\frac{\lambda}{2}}W_{-n-\frac{(\alpha+1)}{2},-\frac{\alpha}{2}}(\lambda)\,,\qquad \lambda\in \C_+$$ where $W_{k,\mu}$ is the Whittaker function, see \cite[p.24]{Badii}.
}

\end{remark}

\section{Laguerre expansions in Banach spaces}

\setcounter{theorem}{0}
\setcounter{equation}{0}

In this section  we study Laguerre expansions on Lebesgue spaces $L^p(\R_+)$ and on abstract Banach spaces $X$. First we show that the span generated by the Laguerre functions $\{\ell_n^{(\alpha)}\}_{n\ge 0}$ are dense in $L^p(\R_+)$ for $\alpha>{-1\over p}$, Theorem \ref{densi} (ii). The fractional semigroup $\{I_\alpha\}_{ \alpha>0}$ and exponential functions $\{e_a\}_{a>0}$  may be expressed via Laguerre expansions, see  Theorem \ref{fract}.
These expansions will be applied to obtain  different representations of operator families related with semigroups theory in the next section. Following the original proof of J. V. Uspensky , we prove a vector-valued theorem of  Laguerre expansions for continuous functions (Theorem \ref{upesvect}).

\begin{theorem} \label{densi} Take $1\leq p < \infty.$
\begin{itemize}
\item[(i)] For $\alpha>-\frac{1}{p}$, $\lambda>0$, and $\xi_{\alpha, \lambda}(t):= t^\alpha e^{-\lambda t}$ for $t>0$, we obtain that
$$\xi_{\alpha, \lambda+1}=\displaystyle\sum_{n=0}^{\infty}\frac{\lambda^n}{(\lambda +1)^{n+\alpha+1}}\frac{\Gamma(n+\alpha+1)}{n!}\ell_n^{(\alpha)} \hbox{ in $L^p(\R_+)$.}$$
\item[(ii)] The set span$\{\ell^{(\alpha)}_n \,\, \vert {n\ge 0}\}$ is dense in $L^p(\R_+)$ if $\alpha >-\frac{1}{p}.$
\end{itemize}
\end{theorem}
\begin{proof} (i) First note that for all $\lambda >0,$ $e^{-\lambda t}=\displaystyle\sum_{n=0}^{\infty}\frac{\lambda^n}{(\lambda +1)^{n+\alpha+1}}L_n^{(\alpha)}(t)$ pointwise, see  formula (\ref{expo2}). Then $$t^{\alpha}e^{-(\lambda+1)t}=\displaystyle\sum_{n=0}^{\infty}\frac{\lambda^n}{(\lambda +1)^{n+\alpha+1}}\frac{\Gamma(n+\alpha+1)}{n!}\ell_n^{(\alpha)}(t), \qquad t>0.$$ Furthermore this convergence is in $L^p(\R_+):$
$$\lVert \displaystyle\sum_{n=0}^{\infty}\frac{\lambda^n}{(\lambda +1)^{n+\alpha+1}}\frac{\Gamma(n+\alpha+1)}{n!}\ell_n^{(\alpha)} \rVert_p\leq C_{p,\alpha}\displaystyle\sum_{n=0}^{\infty}\frac{\lambda^n}{(\lambda +1)^{n+\alpha+1}}n^{\alpha+\frac{1}{2}}<\infty, \qquad \lambda>0,$$
where we have applied Theorem \ref{main2} (iii). (ii) Using the part (i), it is enough to see that span$\{t^{\alpha}e^{-(\lambda+1)t} \,\, \vert {\lambda > 0}\}$ is dense in $L^p(\R_+)$ to get the result. Let $f\in L^q(\R_+)$ with $\frac{1}{p}+\frac{1}{q}=1$ such that  $$\int_0^{\infty}f(t)t^{\alpha}e^{-(\lambda+1)t}\,dt=0, \qquad \lambda >0.$$  By  H\"{o}lder inequality,  the function $f\xi_{\alpha,1}\in L^1(\R_+)$ and  then $$0=\int_0^{\infty}f(t)t^{\alpha}e^{-(\lambda+1)t}\,dt=\mathcal{L}(f\xi_{\alpha,1})(\lambda), \qquad \lambda >0.$$ Since the Laplace transform is injective in $L^1(\R_+)$, we conclude that  $f= 0.$
\end{proof}

\begin{theorem}\label{fract}

 \begin{itemize}
 \item[(i)]  For $2\beta>\alpha>-1,$ the equality $\displaystyle{I_{\beta+1}=\displaystyle\sum_{n=0}^{\infty}\binom{\alpha-\beta+n-1}{n}\ell_n^{(\alpha)}}$  holds in $L^1(\R_+).$
     In particular, for $\alpha>2$, we have that $I_{\alpha}=\displaystyle\sum_{n=0}^{\infty}\ell_{n}^{(\alpha)}\ $ in $L^1(\R_+).$

 \item[(ii)] For all $a>0,$ the equality $\displaystyle{e_a=\displaystyle\sum_{n=0}^{\infty}\frac{(a-\frac{1}{2})^n}{(a+\frac{1}{2})^{n+1}}\mathscr{L}_n^{(0)} }$ holds in $L^1(\R_+).$
 \end{itemize}
\end{theorem}
\begin{proof} (i) As $\displaystyle{\frac{1}{t^{\alpha-\beta}}=\sum_{n=0}^{\infty}\binom{\alpha-\beta+n-1}{n}\frac{n!\Gamma(\beta+1)}{\Gamma(n+\alpha+1)}L_n^{(\alpha)}(t)}$ for all $t>0$ with \mbox{$2\beta>\alpha-1,$} (see \cite[p.89]{Lebedev}) we obtain that $$\frac{t^{\beta}e^{-t}}{\Gamma(\beta+1)}=\displaystyle\sum_{n=0}^{\infty}\binom{\alpha-\beta+n-1}{n}\ell_n^{(\alpha)}(t), \qquad t>0.$$ This convergence is in $L^1(\R_+)$ when $2\beta>\alpha>-1:$ we apply Proposition  \ref{main2} (ii) to get that $$\lVert \displaystyle\sum_{n=0}^{\infty}\binom{\alpha-\beta+n-1}{n}\ell_n^{(\alpha)} \rVert_1\leq \displaystyle\sum_{n=0}^{\infty}C_{\alpha}\frac{n^{\alpha-\beta-1}}{n^{\frac{\alpha}{2}}}.$$

(ii)
Since $\{\mathscr{L}_n^{(0)}\}$ is an orthonormal basis in $L^2(\R_+)$ and the function $e_a\in L^2(\R_+)$ for $a>0$, then the serie $\displaystyle\sum_{n=0}^{\infty}\frac{(a-\frac{1}{2})^n}{(a+\frac{1}{2})^{n+1}}\mathscr{L}_n^{(0)}$  converges to $e_a$ in $L^2(\R_+)$ (in fact in $L^p(\R_+)$ for ${4\over 3}<p<4$, see Introduction and \cite{Askey}) for $a>0$:
$$
\langle e_a, \mathscr{L}_n^{(0)}\rangle=\int_0^\infty e^{-at} e^{-t\over 2}L_n^{(0)}(t)dt= \int_0^\infty e^{-\left(a-{1\over 2}\right)t}\ell_{n}^{(0)}(t)dt= \frac{(a-\frac{1}{2})^n}{(a+\frac{1}{2})^{n+1}},
$$
where we have applied the formula
(\ref{LaplaceLagFunc}).  This convergence also holds in $L^1(\R_+)$:
$$
\lVert\displaystyle\sum_{n=0}^{\infty}\frac{(a-\frac{1}{2})^n}{(a+\frac{1}{2})^{n+1}}\mathscr{L}_n^{(0)}\rVert_1\le
\sum_{n=0}^{\infty}\frac{b^n}{|a+\frac{1}{2}|}\Vert \mathscr{L}_n^{(0)}\Vert_1
\leq C\displaystyle\sum_{n=0}^{\infty}\frac{b^n}{|a+\frac{1}{2}|}\sqrt{n},
$$
with $b=|\frac{a-\frac{1}{2}}{a+\frac{1}{2}}|<1$ for $a>0,$ where we have applied the equivalence (\ref{muck}).  Note that $e^{-at}=\displaystyle\sum_{n=0}^{\infty}\frac{(a-\frac{1}{2})^n}{(a+\frac{1}{2})^{n+1}}\mathscr{L}_n^{(0)}(t)$ pointwise for $t>0$  by Theorem \ref{upes}
 and we conclude the proof.
\end{proof}

Let $X$ be a Banach space, and $f:(0,\infty)\to X$ be a vector-valued continuous function. We say that this function $f$ is differentiable at $t$ if exists $$\displaystyle\lim_{h\to 0^+}\frac{1}{h}(f(t+h)-f(t))$$ in $X.$ Now, we give a vector-valued version of Theorem \ref{upes}. The proof is similar to the  scalar case shown in \cite{Uspensky} and we include the proof to ease the reading.

\begin{theorem}\label{upesvect} Let $X$ be a Banach space and $f:(0,\infty)\to X$ a differentiable function such that the integral $\int_{0}^{\infty}e^{-t}t^{\alpha}\lVert f(t)\rVert^2\,dt$ is finite, then the series $\displaystyle\sum_{n=0}^{\infty}c_n(f)\,L^{(\alpha)}_n(t),$ with $$c_n(f)=\int_{0}^{\infty}\ell_n^{(\alpha)}(t)f(t)\,dt,$$ converges pointwise to $f(t).$
\end{theorem}
\begin{proof} By the Cauchy-Schwartz inequality, we get that  $c_n(f)\in X,$ $$\lVert c_n(f) \rVert\leq (\lVert \mathscr{L}_n^{(\alpha)} \rVert_2)^{\frac{1}{2}}\left(\frac{n!}{\Gamma(n+\alpha+1)}\displaystyle\int_{0}^{\infty}e^{-t}t^\alpha\lVert f(t) \rVert^2dt\right)^{\frac{1}{2}}<\infty, $$ where we have applied that $\{\mathscr{L}_n^{(\alpha)}(t)\}_{n\ge 0}$ is a orthonormal basis in $L^2(\R_+)$ and  $f$ satisfies the hypothesis.

Let $S_m(f)$ be the sum of the  first $m+1$ terms of the series, $$S_m(f)(t):=\sum_{n=0}^m c_n(f)L_n^{(\alpha)}(t)=\displaystyle\int_{0}^{\infty}e^{-y}y^{\alpha}\varphi_{m}(t,y)f(y)\,dy,\qquad t>0,$$ where $\varphi_m(t,y)=\displaystyle{\sum_{n=0}^m\frac{n!}{\Gamma(n+\alpha+1)}L_n^{(\alpha)}(t)L_n^{(\alpha)}(y),}$ for $t,y>0.$ Note that $\varphi_m(t,y)=\varphi_m(y,t)$, for $t,y >0$, $\displaystyle\int_{0}^{\infty}e^{-y}y^\alpha L_m^{(\alpha)}(t,y)\,dy=1$ for $m\ge 0$  and $$\varphi_m(t,y)={(m+1)!\over \Gamma(m+\alpha+1)}\left({L_{m+1}^{(\alpha)}(t)L_m^{(\alpha)}(y)-L_{m+1}^{(\alpha)}(y)L_m^{(\alpha)}(t)\over y-t}\right), \qquad t\not=y >0,$$ see these properties in \cite[p. 611]{Uspensky}. Now, we write \begin{eqnarray*}
&\quad&S_m(f)(t)-f(t)=\int_{0}^{\infty}e^{-y}y^\alpha\varphi_m(t,y)(f(y)-f(t))\,dy, \qquad t>0, \\
\end{eqnarray*}
 and we will conclude that $S_mf(t)-f(t)\to 0$ when $m \to \infty$ for $t>0$. Take $a<t<b$ and there exist $H,G$ such that $0<H<a<b<G$ and
 $$
 \int_0^H y^\alpha e^{-y}\varphi^2_m(y,t)dy <{C}, \qquad \int_G^\infty y^\alpha e^{-y}\varphi^2_m(y,t)dy <{C},
 $$
 for a constant $C>0$, see \cite[p. 614, formula (27)]{Uspensky}.

 Take $\varepsilon>0$. There exist $H,G>0$ in the above conditions such that
 \begin{eqnarray*}
&\quad&\Vert \int_{0}^{H}e^{-y}y^\alpha\varphi_m(t,y)(f(y)-f(t))\,dy\Vert \\&\quad& \qquad \le\left(\int_0^H y^\alpha e^{-y}\varphi^2_m(y,t)dy\right)^{1\over 2}\left(\int_0^H y^\alpha e^{-y}\Vert f(y)-f(t)\Vert^2dy\right)^{1\over 2}\le {\varepsilon\over 3}
\end{eqnarray*}
and similarly
\begin{eqnarray*}
&\quad&\Vert \int_{G}^{\infty}e^{-y}y^\alpha\varphi_m(t,y)(f(y)-f(t))\,dy\Vert \\&\quad& \qquad \le\left(\int_G^\infty y^\alpha e^{-y}\varphi^2_m(y,t)dy\right)^{1\over 2}\left(\int_G^\infty y^\alpha e^{-y}\Vert f(y)-f(t)\Vert^2dy\right)^{1\over 2}\le {\varepsilon\over 3}.
 \end{eqnarray*}
Note that
\begin{eqnarray*}
\varphi_m(t,y)={\sqrt{(m+1)(m+\alpha+1)}\over \pi m}{(ty)^{{-\alpha \over 2}-{1\over 4}}e^{t+y\over 2}\over t-x}\left(T_m(t,y)+ {U_m(t,y)\over \sqrt{m}}\right),
\end{eqnarray*}
where
\begin{eqnarray*}
T_m(t,y)&=& \sqrt{y}\cos\left(2\sqrt{mt}-{2\alpha+1\over 4}\pi\right)\sin\left(2\sqrt{my}-{2\alpha+1\over 4}\pi\right)\\
&\quad&-\sqrt{t}\cos\left(2\sqrt{my}-{2\alpha+1\over 4}\pi\right)\sin\left(2\sqrt{mt}-{2\alpha+1\over 4}\pi\right)
\end{eqnarray*}
and $\vert U_n(t,y)\vert \le M$ for  $y\in (H, G)$, see \cite[pp. 612-613]{Uspensky}. Now we define $$F(t,y)=\frac{f(y)-f(t)}{y-t}, \qquad t\not=y >0.$$ Observe that as function of $y,$  $F(t, \cdot)$ is continuous in $(0,\infty)$ for any $t>0$. Finally, we have that
\begin{eqnarray*}&\quad &\int_H^G e^{-y}y^\alpha\varphi_m(t,y)(f(y)-f(t))\,dy\\
&\quad&\qquad \qquad= C_m\, t^{{-\alpha \over 2}-{1\over 4}}e^{t\over 2}\int_H^G e^{-{y\over 2}}y^{{\alpha \over 2}-{1\over 4}}\left(T_m(t,y)+{U_m(t,y)\over \sqrt{m}}\right)F(t,y)\,dy,
\end{eqnarray*}
where $\displaystyle{\sup_{m\ge 1}C_m<\infty}$.  Note that if $m\to \infty$, then
 \begin{eqnarray*}\displaystyle{\int_H^G e^{-{y\over 2}}y^{{\alpha \over 2}-{1\over 4}}{U_m(t,y)\over \sqrt{m}}F(t,y)\,}dy\to 0,\qquad
\displaystyle{
\int_H^G e^{-{y\over 2}}y^{{\alpha \over 2}-{1\over 4}}T_m(t,y)F(t,y)\,dy} \to 0,
 \end{eqnarray*} where we have applied the Riemann-Lebesgue Lemma in the second limit, see for example \cite[Theorem 1.8.1 c)]{ABHN}.  We conclude that $$\displaystyle{\Vert \int_H^G e^{-y}y^\alpha\varphi_m(t,y)(f(y)-f(t))\,dy\Vert \le {\varepsilon \over 3}},$$ and $ \displaystyle\lim_{m\to\infty}\lVert S_m(f)(t)-f(t)\rVert=0,$ for $t>0$.\end{proof}

\begin{remark}\label{UMD} {\rm In fact,  the original theorem of  J. V. Uspensky is proved requiring less restrictive conditions on the function $f$, in particular, the existence of singular Dirichlet  integral, see \cite[p. 617]{Uspensky}.  In UMD Banach spaces,  the  convergence of the Laguerre expansion for  continuous functions might be proved following the original proof given in \cite{Uspensky}.

On the other hand, a straightforward application of Theorem \ref{upes} allows to obtain the weak convergence of the partial series $S_m(f)(t)$ to the function $f(t)$ for $t>0$.}
\end{remark}

\section{$C_0$-semigroups and resolvent operators given by Laguerre expansions}

\setcounter{theorem}{0}
\setcounter{equation}{0}

In this section, we are interested in representing $C_0$-semigroups and resolvent operators on series of Laguerre polynomials. To do this, we will apply several results included in  previous sections. However,  first of all,  we give some basic results from $C_0$-semigroup theory,   for further details see monographies \cite{ABHN, Nagel}.

Given $-A$ a closed operator  in a Banach space $X.$ The resolvent operator $\lambda\to (\lambda+A)^{-1}$ is analytic in the resolvent set $\rho(-A)$, and $$\frac{d^n}{d\lambda^n}(\lambda+A)^{-1}=(-1)^n n! (\lambda+A)^{-n-1}\ \text{for all }n\in\N,$$ see \cite[p.240]{Nagel}. We say that a $C_0$-semigroup $(T(t))_{t\geq 0}$ is uniformly bounded if $\lVert T(t)\rVert\leq M$ for all $t\geq 0,$ with $M$ a positive constant. Let $-A$ be the infinitesimal generator of a uniformly bounded $C_0$-semigroup $(T(t))_{t\geq 0}$. For $\alpha>0$ and $\lambda>0$ we define the fractional powers of the resolvent operator as below \begin{equation}\label{4}(\lambda+A)^{-\alpha}x:=\frac{1}{\Gamma(\alpha)}\int_0^{\infty}t^{\alpha-1}e^{-\lambda t}T(t)x\,dt,\qquad x\in X,\end{equation} see \cite[Proposition 11.1]{Komatsu}. The Cayley transform of $-A,$ i.e., $V:=(A-1)(A+1)^{-1},$  defines a bounded operator  called the cogenerator of the $C_0$-semigroup, see for example \cite{Eisner, Gomilko}. Note that \begin{equation}\label{coge}A^n(A+1)^{-n-\alpha-1}x=\biggl(\frac{V+1}{2}\biggr)^n\biggl(\frac{1-V}{2}\biggr)^{\alpha+1}x, \qquad x\in X.
\end{equation}
Nice identities between the powers of the cogenerator and integral expressions which involve generalized Laguerre polynomials appear in \cite[Theorem 1]{Gomilko}, \cite[Lemma 4.4]{BZ} and \cite[Lemma 6.7]{Besselin}.

\begin{theorem}\label{loc} Let $(T(t))_{t\geq 0}$ be a uniformly bounded $C_0$-semigroup in a Banach space $X$ with infinitesimal generator $(-A, D(A)).$
 \begin{itemize}

 \item[(i)]For $n\in \NN\cup\{0\}$ and $\alpha >-1$,
 $$
  \int_{0}^{\infty}\ell_n^{(\alpha)}(t)T(t)x\,dt=A^n(A+1)^{-n-\alpha-1}x, \qquad x\in X.
 $$

 \item[(ii)] For $n\in \NN\cup\{0\}$ and $\alpha>0$,
 $$
  \int_{0}^{\infty}\ell_n^{(\alpha)}(t)(t+A)^{-1}x\,dt=\int_{0}^{\infty}\varphi_{n, \alpha}(t)T(t)x\,dt, \qquad x\in X.
 $$

 \item[(iii)] For $x\in D(A)$ and $\alpha>-1$, $$T(t)x=\displaystyle\sum_{n=0}^{\infty}A^n(A+1)^{-n-\alpha-1}x\,L^{(\alpha)}_n(t), \qquad t>0.$$

\item[(iv)]For $x\in D(A)$ and $\alpha>-1$, $$T(t)x=\biggl(\frac{1-V}{2}\biggr)^{\alpha+1}\displaystyle\sum_{n=0}^{\infty}\biggl(\frac{V+1}{2}\biggr)^nx\,L^{(\alpha)}_n(t), \qquad t>0.$$

     \item[(v)] For $x\in X$ and $\alpha >1$, $$(\lambda+A)^{-1}x=\displaystyle\sum_{n=0}^{\infty}\left(\int_{0}^{\infty}\varphi_{n, \alpha}(t)T(t)xdt\right)L^{(\alpha)}_n(\lambda), \qquad \lambda>0.$$
 \end{itemize}

\end{theorem}
\begin{proof} (i)  For $\alpha>-1 $ and $n\in \NN\cup\{0\}$, we have that

$$\int_{0}^{\infty}\ell^{(\alpha)}_n(t)T(t)x\,dt=\frac{1}{\Gamma(n+\alpha+1)}\int_{0}^{\infty}\frac{d^n}{dt^n}(e^{-t}t^{n+\alpha})T(t)x\,dt, \qquad x\in X.$$ We integrate by parts  to obtain that
 $$\int_{0}^{\infty}\ell^{(\alpha)}_n(t)T(t)x\,dt= A^n\frac{1}{\Gamma(n+\alpha+1)}\int_{0}^{\infty}e^{-t}t^{n+\alpha}T(t)x\,dt=A^n(A+1)^{-n-\alpha-1}x,
 $$
where we use the formula \eqref{4}.

 (ii) Take $x\in X$ and $\alpha>0$. The integral  $\int_{0}^{\infty}\ell_n^{(\alpha)}(t)(t+A)^{-1}x\,dt$ converges  due to $\Vert (t+A)^{-1}\Vert \le {M\over t}$ for $t>0$.  Then
 \begin{eqnarray*}
 \int_{0}^{\infty}\ell_n^{(\alpha)}(t)(t+A)^{-1}x\,dt
&=&\frac{1}{\Gamma(n+\alpha+1)}\int_{0}^{\infty}\frac{d^n}{dt^n}(e^{-t}t^{n+\alpha})(t+A)^{-1}x\,dt \\
&=&\frac{1}{\Gamma(n+\alpha+1)}\int_{0}^{\infty}e^{-t}t^{n+\alpha}(\int_0^{\infty}s^ne^{-ts}T(s)x\,ds)\,dt \\
&=&\int_0^{\infty}\frac{s^n}{(s+1)^{n+\alpha+1}}T(s)x\,ds=\int_0^{\infty}\varphi_{n,\alpha}(s)T(s)x\,ds.
\end{eqnarray*}
(iii) For each $x\in D(A),$ the function $T(\cdot)x:\R_+\to X$ is differentiable at every point and ${d\over dt}T(t)x=-T(t)Ax$, see \cite[Definition 1.2, Chapter II]{Nagel}. In addition  note that  $$\int_{0}^{\infty}e^{-t}t^{\alpha}\lVert T(t)x\rVert^2\,dt\leq M^2\int_{0}^{\infty}e^{-t}t^{\alpha}\lVert x\rVert^2\,dt<\infty.$$ Then, we apply Theorem \ref{upesvect} to  have that $\lVert T(t)x-\displaystyle\sum_{n=0}^{m}c_n(T(\cdot)x)\,L^{(\alpha)}_n(t) \rVert\to 0$ as $m\to\infty,$ for all $t> 0,$ with $$c_n(T(\cdot)x)=\int_{0}^{\infty}\ell^{(\alpha)}_n(t)T(t)x\,dt=A^n(A+1)^{-n-\alpha-1}x,$$ where we have applied the part (i). (iv) The proof of (iv) is a straightforward consequence of (iii) and (\ref{coge}). (v) Note that $\displaystyle\int_0^{\infty}e^{-t}t^{\alpha}\lVert (t+A)^{-1}\rVert^2 dt\leq M^2\int_0^{\infty}\frac{e^{-t}}{t^{2-\alpha}}dt<\infty$ if only if $\alpha>1.$ In this case  $(\lambda+A)^{-1}x=\displaystyle\sum_{n=0}^{\infty}c_n((\cdot+A)^{-1}x)\,L^{(\alpha)}_n(\lambda)$ for $x\in X$ with $$
c_n((\cdot+A)^{-1}x)= \int_{0}^{\infty}\ell_n^{(\alpha)}(t)(t+A)^{-1}x\,dt=\int_{0}^{\infty}\varphi_{n, \alpha}(t)T(t)x\,dt,
$$
where we have applied the part (ii).
\end{proof}

For $C_0$-semigroups which are not uniformly bounded, we may perturb the infinitesimal generator to obtain uniformly bounded $C_0$-semigroups, and in this way, $-A-w$  generates a uniformly bounded $C_0$-semigroup, see for example \cite[p.60]{Nagel}.

\begin{corollary} Let $(T(t))_{t\geq 0}$ be a $C_0$-semigroup in a Banach space $X$ with infinitesimal generator $(-A, D(A))$ and exponential bound $w,$ that is, $\lVert T(t)\rVert\leq Me^{wt}.$ Then for \mbox{$x\in D(A)$} and $\alpha>-1$, $$T(t)x=e^{wt}\displaystyle\sum_{n=0}^{\infty}(A+w)^n(A+w+1)^{-n-\alpha-1}x\,L^{(\alpha)}_n(t), \qquad t>0.$$
\end{corollary}

Approximation theory of $C_0$-semigroups has a great importance in many mathematical areas, as PDEs, mathematical physics and probability
theory. There are a large number of results about approximation of $C_0$-semigroups using different approximations. A nice summary of these different approximations may be found in the recent paper \cite{Gomilko}. Subdiagonal Pad\'{e} approximations to  $C_0$-semigroups are given  in \cite{BT, Egert, Neubrander}.

Now we are interested in the approximation order of the $C_0$-semigroup by the $m-$th partial sum of the Laguerre series, where, $$T_{m,\alpha}(t)x:=\displaystyle\sum_{n=0}^{m}A^n(A+1)^{-n-\alpha-1}x\,L^{(\alpha)}_n(t), \qquad x\in X, \quad t> 0.$$ Our  approximation series are computationally better than the rational approximations (for example Euler approximation or Pad\'{e} approximation) since in each step  one has to start to calculate  all approximations over again  (\mbox{$(1+t\frac{A}{m})^{-m}x,$ in the case of Euler approximation)} while for the Laguerre series one have already calculated until \mbox{$A^{m-1}(A+1)^{-(m-1)-\alpha-1}x,$} see \cite{Grimm}.


\begin{theorem}\label{rate} Let $(T(t))_{t\geq 0}$ be a uniformly bounded $C_0$-semigroup in a Banach space $X,$ $\lVert T(t)\rVert\leq M,$ with infinitesimal generator $-A,$ $\alpha>-1$.

 \begin{itemize}

\item[(i)] For $x\in D(A^p)$ and $n\geq p,$ $$\lVert A^n(A+1)^{-n-\alpha-1}x\rVert\leq \frac{C} {n^{\frac{\alpha+p}{2}}}\lVert A^p x \rVert,$$ with $C$ a positive constant depending on $\alpha,$ $p$ and $M.$

    \item[(ii)] For each $t>0$ there is a $m_0\in\N$ such that for all integer $2<p\leq m+1$ with $m\geq m_0,$  $$\lVert T(t)x-T_{m,\alpha}(t)x \rVert\leq \frac{ C_{t,p} \lVert A^p x\rVert}{m^{\frac{p}{2}-1}} \qquad x\in D(A^p),$$ where $C_{t,p}$ is a constant which depends only on $t>0$ and $p$.

 \end{itemize}

\end{theorem}
\begin{proof}
(i) We write $B(x):=A^n(A+1)^{-n-\alpha-1}x=\int_0^{\infty}\ell_n^{(\alpha)}(t)T(t)x\,dt$ for $x\in X,$ see Theorem \ref{loc} (i). We apply Proposition \ref{propLag} (vi) and integrate it by parts to obtain that $$B(x)=\int_0^{\infty}\frac{d^p}{dt^p}(\ell_{n-p}^{(\alpha+p)}(t))T(t)x\,dt=\int_0^{\infty}\ell_{n-p}^{(\alpha+p)}(t)A^pT(t)x\,dt \qquad x\in D(A^p),$$
where we have used that $\frac{d^{p-j}}{dt^{p-j}}(\ell_{n-p}^{(\alpha+p)}(t))_{t=0}=0,$ and $\frac{d^{p-j}}{dt^{p-j}}(\ell_{n-p}^{(\alpha+p)}(t))_{t=\infty}=0,$ for $j=1,2,\ldots,p.$ By Theorem \ref{main2} (ii), $$\lVert B(x) \rVert\leq \frac{C} {n^{\frac{\alpha+p}{2}}}\displaystyle\sup_{ t\ge 0}\lVert A^p T(t)x \rVert < \frac{C} {n^{\frac{\alpha+p}{2}}}\lVert A^p x \rVert,$$  and we conclude the result.

(ii) By Muckenhoupt estimates,
 we know that there is a $m_0\in\N$ such that  $|L_m^{(\alpha)}(t)|\leq C_t m^{\frac{\alpha}{2}},$ for all $m\geq m_0,$ see  (\ref{mucklag}). So, for $m\geq m_0$ and $x\in D(A^p)$ with $p\leq m+1,$ we apply Theorem \ref{loc} (iii) and (i) to get that  $$
\lVert T(t)x-T_{m,\alpha}(t)x\rVert\leq\displaystyle\sum_{n=m+1}^{\infty}\lVert A^n(A+1)^{-n-\alpha-1}x \rVert |L_n^{(\alpha)}(t)|\leq\displaystyle\sum_{n=m+1}^{\infty}\frac{C_t}{n^{\frac{p}{2}}}\lVert A^p x\rVert\leq\frac{ C_{t,p} \lVert A^p x\rVert}{m^{\frac{p}{2}-1}},$$
where we have used the bound of part (i).
\end{proof}

The following result  gives  some representation formulae (via series of operators) of fractional powers of resolvent operator of $-A$. The Hille functional calculus is a linear and bounded map $f\mapsto f(A)$, $L^1(\R_+) \to {\mathcal B}(X)$, where
$$
f(A)x:=\int_0^\infty f(t)T(t)xdt, \qquad x\in X, \quad f\in L^1(\R_+),$$
and $(T(t))_{t\geq 0}$ is a uniformly bounded semigroup generated by $(-A, D(A))$ (see for example \cite[Chapter 3]{Sinclair}). Note that if $f= \sum_{n=1}^\infty f_n$ in $L^1(\R_+)$ then  $f(A)= \sum_{n=1}^\infty f_n(A)$ in ${\mathcal B}(X)$.

\begin{theorem}\label{appr} Let $(-A, D(A))$ be the infinitesimal generator of a uniformly bounded $C_0$-semigroup $(T(t))_{t\geq 0}$ in a Banach space $X.$
\begin{itemize}
\item[(i)] For  $2\beta>\alpha>-1,$ and $x\in X$,
$$(A+1)^{-\beta-1}=\displaystyle\sum_{n=0}^{\infty}\binom{\alpha-\beta+m-1}{m}A^n(A+1)^{-n-\alpha-1} \hbox{ in }{\mathcal B}(X).$$ In particular for $\alpha>2,$ we have that
$(A+1)^{-\alpha}=\displaystyle\sum_{n=0}^{\infty}A^n(A+1)^{-n-\alpha-1} \hbox{ in }{\mathcal B}(X).$

\item[(ii)] For all $a>0$, $(a+A)^{-1}=\displaystyle\sum_{n=0}^{\infty}\frac{(a-\frac{1}{2})^n}{(a+\frac{1}{2})^{n+1}}(A-\frac{1}{2})^n(A+\frac{1}{2})^{-n-1}$ holds in ${\mathcal B}(X).$

\end{itemize}
\end{theorem}
\begin{proof}
(i) By Theorem \ref{fract} (i) and Theorem \ref{loc} (i), we have that  \begin{eqnarray*}
(A+1)^{-\beta-1}x&=&\frac{1}{\Gamma(\beta+1)}\int_0^{\infty}t^{\beta}e^{-t}T(t)x\,dt=\displaystyle\sum_{n=0}^{\infty}\binom{\alpha-\beta+n-1}{n}\int_0^{\infty}\ell_n^{(\alpha)}(t)T(t)x\,dt \\
&=&\displaystyle\sum_{n=0}^{\infty}\binom{\alpha-\beta+n-1}{n}A^n(A+1)^{-n-\alpha-1}x.
\end{eqnarray*}

(ii) For $a>0$ and $x\in X,$ we have by  Theorem \ref{fract} (ii),\begin{eqnarray*} (a+A)^{-1}x&=&\int_0^{\infty}e^{-at}T(t)x\,dt=\displaystyle\sum_{n=0}^{\infty}\frac{(a-\frac{1}{2})^n}{(a+\frac{1}{2})^{n+1}}\int_0^{\infty}e^{-\frac{t}{2}}L_n^{(0)}(t)T(t)x\,dt \\
&=&\displaystyle\sum_{n=0}^{\infty}\frac{(a-\frac{1}{2})^n}{(a+\frac{1}{2})^{n+1}}\int_0^{\infty}\ell_n^{(0)}(t)S(t)x\,dt,
\end{eqnarray*}
with $S(t):=e^{\frac{t}{2}}T(t)$is the $C_0$-semigroup generated by $\frac{1}{2}-A.$ By Theorem \ref{fract} (i), we conclude that  $ \displaystyle{(a+A)^{-1}x=\displaystyle\sum_{n=0}^{\infty}\frac{(a-\frac{1}{2})^n}{(a+\frac{1}{2})^{n+1}}(A-\frac{1}{2})^n(A+\frac{1}{2})^{-n-1}x},$ for $x\in X.$
\end{proof}



\section{Examples and final comments}

In this section we apply our results to concrete examples: translation, convolution  and multiplication semigroups in Lebesgue space. Holomorphic semigroups allow us to improve some previous results, compare Theorem \ref{Holomorphic} and Theorem \ref{loc}.

\setcounter{theorem}{0}
\setcounter{equation}{0}

\subsection{Translation semigroup}
Let $L^p(\R_+)$  be with $1\leq p<\infty.$ The translation semigroup (or shift semigroup) in $L^p(\R_+)$, $(T(t))_{t\geq 0},$ defined by \begin{displaymath}
T(t)f(x):=(\delta_t\ast f)(x)=\left\{\begin{array}{ll}
f(x-t),&x>t,\\
0,&x\leq t,
\end{array} \right.
\end{displaymath} is an isometry $C_0$-semigroup. The infinitesimal generator $-A$ is the usual derivation operator, $-A=-\frac{d}{dx}.$ Furthermore, $
(A+1)^{-n-\alpha-1}f=I_{n+\alpha +1}*f,$ and
$A^n(A+1)^{-n-\alpha-1}f=\ell^{(\alpha)}_n*f,$
for $n\in\NN\cup\{0\}$ and $\alpha >-1$. By Theorem \ref{loc} (iii), we obtain the formula
$$
\delta_t\ast f=\displaystyle\sum_{n=0}^{\infty}(\ell_n^{(\alpha)}\ast f)L_n^{(\alpha)}(t), \qquad f\in W^{(1),p}(\R_+),
$$
where $W^{(1),p}(\R_+)$ is the Sobolev space defined by \mbox{$W^{(1),p}(\R_+)=\{f\in L^p(\R_+)\,\,|\,\,f'\in L^p(\R_+)\}$.} In particular for $\alpha=0$ this formula has been considered in \cite[Section 3, Examples 3.1(4)]{Du} where Laguerre expansions of tempered distributions are studied.

\subsection{Convolution and multiplication semigroups}
Let $L^p(\R^m)$  with $m\geq 1$ and $1\leq p<\infty.$ For $t>0$, let \mbox{$k_t:\R^m\to\R$} be a convolution kernel on $L^p(\R^m)$ such that  $\displaystyle\sup_{t>0}\Vert k_t\Vert_1<\infty$ and $(T(t))_{t\geq 0},$ defined by $$T(t)f(s):=(k_t*f)(s)= \int_{\R^m}k_t(s-u)f(u)du, \qquad f\in L^p(\R^m),$$ is a uniformly bounded $C_0$-semigroup, whose generator is denoted by $-A$.

Fixed $s\in\R^m,$ we suppose that the map  $t \mapsto k_t(s)$ is differentiable with respect to $t$ and
\begin{equation}\label{formu}\int_{0}^{\infty}e^{-t}t^{\alpha}|k_t(s)|^2\,dt<\infty.
 \end{equation}

  By  Theorem \ref{upes}, we have that $k_t(s)=\displaystyle\sum_{n=0}^{\infty}a_n(s)L_n^{(\alpha)}(t)$ for $s\in \R^m,$ and $t>0,$ where $$a_n(s)=\int_{0}^{\infty}\ell_n^{(\alpha)}(t)k_t(s)\,dt, \qquad s\in \R^m.$$ Furthermore, $a_n\in L^1(\R^n)$  and
 $\Vert a_n\Vert_1\leq \displaystyle{\sup_{t>0}}\Vert k_t\Vert_1\, \Vert \ell_n^{(\alpha)}\Vert_1.$
   By Fubini's Theorem and Theorem \ref{loc} (i), we obtain that $$A^n(A+1)^{-n-\alpha-1}f=a_n*f, \qquad f\in L^p(\R^m),$$ and therefore by Theorem \ref{loc} (iii), we have $\displaystyle{k_t*f=\displaystyle\sum_{n=0}^{\infty}(a_n\ast f)L_n^{(\alpha)}(t)}$ for $f\in D(A).$

Examples  of convolutions semigroups are the Gaussian and Poisson kernels, $$g_t(s):=\frac{1}{(4\pi t)^{\frac{m}{2}}}e^{-\frac{\lVert s\rVert^2}{4t}}, \qquad p_t(s):=\frac{\Gamma(\frac{m+1}{2})}{\pi^{\frac{m+1}{2}}}\frac{t}{(t^2+\lVert s\rVert^2)^{\frac{m+1}{2}}}$$ see \cite[p.25]{Sinclair}. The condition (\ref{formu}) is verified by Gaussian and Poisson kernels if $s\neq 0$ for $\alpha>-1;$ in the case if $s=0,$ it is needed to consider $\alpha>m-1$ for Gaussian kernel and $\alpha>2m-1$ for Poisson kernel.

Now, let $q:\R^m\to\R_-$ be a continuous function, where $\R_{-}:=(-\infty,0]$. The family of operators in $C_0(\R^m)$, $(S_q(t))_{t\geq 0},$ defined by $S_q(t)f:=e^{tq}f,$ for $t>0,$ is a contraction $C_0$-semigroup whose infinitesimal generator $(-B, D(B))$ is given by  $-Bf:=qf$ and $\ D(B)=\{f\in C_0(\R^m)\,\,|\,\, qf\in C_0(\R^m)\}$.  In these cases, note that $$ B^n(B+1)^{-n-\alpha-1}f(s)=\frac{(-q)^n(s)}{(1-q(s))^{n+\alpha+1}}f(s), \qquad f\in C_0(\R^m), \quad s\in \R^m,$$ and $$S_q(t)f=\displaystyle\sum_{n=0}^{\infty}\frac{(-q)^n}{(1-q)^{n+\alpha+1}}f\,L_n^{(\alpha)}(t),\qquad f\in D(B),\quad t>0.$$

Some examples are the Fourier transforms of the Gaussian and Poisson semigroups, $q(s)=-\lVert s\rVert^2$ and $q(s)=-\lVert s\rVert$ (\cite[p.69]{Nagel}); or subordinated semigroups to them with $q(s)=-\text{log}(1+ \lVert s\rVert^2)$ and \mbox{$q(s)=-\text{log}(1+\lVert s\rVert),$} studied in details in \cite{Campos}.

To finish this subsection we are interested in identifying the functions $a_n$ in the cases of Poisson and Gaussian semigroup. We denote by \mbox{$\mathcal{F}: L^1(\R^m)\to C_0(\R^m)$} the usual Fourier transform, defined by
$\displaystyle \mathcal{F}(f)(s):=\int_{\R^m} e^{-i s\cdot u}f(u)\,du,$
 with $s\cdot u$ the inner product in $\R^m.$ For the Poisson semigroup, we have $\mathcal{F}(a_n)(s)=\varphi_{n,\alpha}(\lVert s\rVert)$ for $s\in \R^m$ and \mbox{$\mathcal{F}(a_n)\in L^1(\R^m)$} for $\alpha>m-1$ and belongs to $L^2(\R^m)$ for $\alpha>\frac{m-2}{2}$, see functions $\varphi_{n,\alpha}$ defined in (\ref{definit}). For the Gaussian semigroup, $\mathcal{F}(a_n)(s)=\varphi_{n,\alpha}(\lVert s\rVert^2)$  for $s\in \R^m$ and \mbox{$\mathcal{F}(a_n) \in L^1(\R^m)$} for $\alpha>\frac{m-2}{2}$ and belongs to $L^2(\R^m)$ for $\alpha>\frac{m-4}{4}.$ To show both results, we use spherical coordinates and apply Theorem \ref{phi}\ (i). Note that in these cases $a_n$ is a radial function.  Then $\mathcal{F}(a_n)=2\pi\mathcal{F}^{-1}(a_n)$ and $$a_n(s)=\frac{C_m}{(2\pi)^m}\int_0^{\infty}\mathcal{F}(a_n)(r)j_m(\lVert s\rVert r)\,r^{m-1}\,dr,\qquad s\in \R^m,$$ where $C_m$ is the area of the unit $(m-1)$-dimensional sphere and $j_m$
is the spherical Bessel function for $m\geq2$, see more details in \cite[p.133]{Dym}. In the particular case of $m=1,$ we have that $$a_n(s)=\frac{1}{2\pi}\int_0^{\infty}e^{i s r}\mathcal{F}(a_n)(r)\,dr+\frac{1}{2\pi}\int_0^{\infty}e^{- i s r}\mathcal{F}(a_n)(r)\,dr,$$ since $\mathcal{F}(a_n)$ is an even function. For $\alpha>0$  and the Poisson semigroup, we have that \begin{eqnarray*}
a_n(s)&=&\frac{1}{2\pi}\biggl(\displaystyle\lim_{\lambda\to i s}\mathcal{L}(\varphi_{n,\alpha})(\lambda)+\displaystyle\lim_{\lambda\to - i s}\mathcal{L}(\varphi_{n,\alpha})(\lambda)\biggr) \\
&=&\frac{n!}{2\pi}\biggl((-is)^{\frac{\alpha-1}{2}}e^{\frac{-i s}{2}}W_{-n-\frac{(1+\alpha)}{2},-\frac{\alpha}{2}}(-i s)+(is)^{\frac{\alpha-1}{2}}e^{\frac{i s}{2}}W_{-n-\frac{(1+\alpha)}{2},-\frac{\alpha}{2}}(i s)\biggr),
\end{eqnarray*}
for $s\in \R$ and $W_{k,\mu}$ is the Whittaker function; for the Gaussian semigroup, $$a_n(s)=\frac{1}{2\pi}\biggl(\displaystyle\lim_{\lambda\to i s}\mathcal{L}(\varphi_{n,\alpha}(t^2))(\lambda)+\displaystyle\lim_{\lambda\to -i s}\mathcal{L}(\varphi_{n,\alpha}(t^2))(\lambda)\biggr),\qquad s\in \R.$$
To calculate $\mathcal{L}(\varphi_{n,\alpha}(t^2))(\lambda)$ with  $\lambda \in \C_+$, note that \begin{eqnarray*}
\mathcal{L}(\varphi_{n,\alpha}(t^2))(\lambda)=\frac{1}{\sqrt{\pi}}\int_0^{\infty}e^{-\frac{\lambda^2}{4 u^2}}\mathcal{L}(\varphi_{n,\alpha})(u^2)\,du
=\frac{n!}{\sqrt{\pi}}\int_0^{\infty}e^{-\frac{\lambda^2}{4 u^2}}u^{\alpha-1}e^{\frac{u^2}{2}}W_{-n-\frac{(1+\alpha)}{2},-\frac{\alpha}{2}}(u^2)\,du,
\end{eqnarray*}
where we have used Laplace transform properties and Remark \ref{whittaker}.

\subsection{Differentiable and holomorphic semigroups}

So far, we have considered $C_0$-semigroups defined in $[0,\infty).$ If we consider differentiable or holomorphic semigroups with certain growth assumptions, some previous results will be improved.

A  $C_0$-semigroup $(T(t))_{t\ge 0}$ is called immediately differentiable if the orbit $t\mapsto T(t)x$ is differentiable for $t>0$, see \cite[Definition II.4.1]{Nagel}.  A straightforward consequence of Theorem \ref{loc} (iii) is the following corollary. A similar result seems that it may hold in UMD spaces for every uniformly bounded $C_0$-semigroup, see Remark \ref{UMD}.

\begin{corollary}Let $-A$ be the infinitesimal generator of a immediately differentiable uniformly bounded   $C_0$-semigroup $(T(t))_{t\geq 0}$ in a Banach space $X.$ Then
$$T(t)x=\sum_{n=0}^{\infty}A^n(A+1)^{-n-\alpha-1}x\,L^{(\alpha)}_n(t),\qquad t>0, \quad  x\in X.$$

\end{corollary}

A holomorphic semigroup $(T(z))_{z\in\C_+}$ is said of type $HG_{\beta}$ if $\displaystyle\lVert T(z) \rVert\leq C_{\nu}\biggl(\frac{|z|}{\Re z}\biggr)^{\nu},$ for $\Re z>0$ and every $\nu>\beta,$ see definition in \cite{Gale}. Classical holomorphic semigroup as the Gaussian and the Poisson semigroup satisfy the $HG_{\beta}$ condition for some $\beta>0$. Moreover, a large amount of different semigroups as those generated by $-\sqrt{\mathcal{L}},$ where $\mathcal{L}$ is the sub-Laplacian in $L^p(\mathbb{G})$ and $\mathbb{G}$ a Lie Group, or generated by $-(-\log(\lambda)+H),$ where $H$ is the strongly elliptic operator affiliated with a strongly continuous representation of a Lie Group into a Banach space, satisfy the $HG_{\beta}$ condition and can be found in \mbox{\cite[Section 5]{Gale2}.}

\begin{theorem}\label{Holomorphic} Let $-A$ be the infinitesimal generator of a uniformly bounded holomorphic $C_0$-semigroup $(T(z))_{z\in\C_+}$ of type $HG_{\beta}$ for some $\beta\ge 0$ in a Banach space $X.$

\begin{itemize}
\item[(i)] For $\alpha >-1$, there exists $C_{\alpha,\beta} >0$ such that
 $$\lVert A^n(A+1)^{-n-\alpha-1}x \rVert\leq \frac{C_{\alpha,\beta}\Vert x\Vert}{n^{\alpha+1}}, \qquad x\in X,$$

\item[(ii)] For $\alpha>0$ and $t>0,$  $$\lVert T(t)x-\displaystyle\sum_{n=0}^{m}A^n(A+1)^{-n-\alpha-1}x\,L^{(\alpha)}_n(t) \rVert\leq \frac{C_{\alpha,\beta,t}}{m^{\frac{\alpha}{2}}}\Vert x\Vert, \qquad x\in X.$$
\end{itemize}

\end{theorem}
\begin{proof} (i)  By Theorem \ref{phi} (iii), $\varphi_{n,\alpha}\in AC^{(j)}$ for $j,n\in\N,$ and $\alpha>-1$. We apply holomorphic functional calculus $f\mapsto f(A)$ defined in \cite[Theorem 6.2.]{Gale} to get that
 $$\lVert A^n(A+1)^{-n-\alpha-1}x \rVert= \lVert \varphi_{n,\alpha}(A)x \rVert\leq C\Vert x\Vert \,\Vert \varphi_{n,\alpha}\Vert_{(j)}\leq \frac{C_{\alpha,\beta}\Vert x\Vert}{n^{\alpha+1}},\qquad x\in X,
 $$
 for all $j>\beta+1$ and $x\in X$. (ii) For $t>0$ there is a $n_0\in\N$ such that the inequality $|L_n^{(\alpha)}(t)|\leq C_tn^{\frac{\alpha}{2}}$ holds for  $n\geq n_0$, see formula (\ref{mucklag}).
 We use part  (i) to get that $$
\lVert T(t)x-\displaystyle\sum_{n=0}^{m}A^n(A+1)^{-n-\alpha-1}x\,L^{(\alpha)}_n(t)\rVert \leq \displaystyle\sum_{n=m+1}^{\infty}\frac{C_{\alpha,\beta,t}}{n^{1+\frac{\alpha}{2}}}\Vert x\Vert\leq\frac{C_{\alpha,\beta,t}}{m^{\frac{\alpha}{2}}}\lVert x\rVert, \quad x\in X,$$
and we conclude the result.
\end{proof}

\subsection*{Acknowledgements}

The authors wish to thank O. Blasco, M. L. Rezola and L. Roncal for the pieces of advice and assistance provided in order to obtain some previous results.

\end{document}